\documentclass[reqno]{amsart}
\usepackage{amsfonts}
\usepackage{amsmath}
\usepackage{amsthm, amscd}
\usepackage{mathtools}
\usepackage{amssymb}
\usepackage{mathrsfs}
\usepackage{graphicx}
\usepackage{caption}
\usepackage{subcaption}

\usepackage[alphabetic]{amsrefs}
\usepackage{etex}
\usepackage{hyperref}
\hypersetup{
  colorlinks = true,
  linkcolor  = black
}

\usepackage{tikz}
\usepackage{xypic}

\allowdisplaybreaks

\theoremstyle{plain}\newtheorem{Theorem}{Theorem}[section]
\theoremstyle{plain}\newtheorem{Corollary}[Theorem]{Corollary}
\theoremstyle{plain}\newtheorem{Lemma}[Theorem]{Lemma}
\theoremstyle{plain}\newtheorem{Definition}[Theorem]{Definition}
\theoremstyle{plain}\newtheorem{Proposition}[Theorem]{Proposition}
\theoremstyle{plain}
\theoremstyle{plain}\newtheorem*{Theorem*}{Theorem}
\theoremstyle{remark}\newtheorem{remark}[Theorem]{Remark}

\numberwithin{equation}{section}

\DeclareMathOperator{\rank}{rank}
\DeclareMathOperator{\II}{I}

\DeclareMathOperator{\AHI}{AHI}

\DeclareMathOperator{\Kh}{Kh}
\DeclareMathOperator{\Khr}{Khr}

\DeclareMathOperator{\lk}{lk}
\DeclareMathOperator{\HFK}{\widehat{HFK}}
\DeclareMathOperator{\HFL}{\widehat{HFL}}

\newcommand{\bC}{\mathbb{C}}

\newcommand{\bQ}{\mathbb{Q}}

\newcommand{\bZ}{\mathbb{Z}}

\author{Yi Xie}
\address{Beijing International Center for Mathematical Research, Peking University, Beijing 100871, China}
\email{yixie@pku.edu.cn}
\author{Boyu Zhang}
\address{Department of Mathematics, Princeton University, New Jersey 08544, USA}
\email{bz@math.princeton.edu}
\title{On links with Khovanov homology of small ranks}

\begin{document}

\begin{abstract}
We classify all links whose Khovanov homology have ranks no greater than $8$, and all three-component links whose Khovanov homology have ranks no greater than $12$, where the coefficient ring is $\bZ/2$. The classification is based on the previous results of Kronheimer-Mrowka \cite{KM:Kh-unknot}, Batson-Seed \cite{Kh-unlink}, Baldwin-Sivek \cite{BS}, and the authors \cite{XZ:forest}.
\end{abstract}

\maketitle

\section{Introduction}

Khovanov homology \cite{Kh-Jones} is a combinatorially defined invariant for oriented links in $S^3$. For a  commutative ring $R$ and an oriented link $L$, the Khovanov homology assigns a bi-graded $R$--module $\Kh(L;R)$. 

The detection properties of Khovanov homology have been studied intensively in the past decade. In 2011, Kronheimer and Mrowka \cite{KM:Kh-unknot} proved that Khovanov homology detects the unknot. Since then, many other detection results of Khovanov homology have been obtained. It is now known that Khovanov homology detects the unlink \cite{Kh-unlink, HN-unlink}, the trefoil \cite{BS}, the Hopf link \cite{BSX}, the forest of unknots \cite{XZ:forest}, the splitting of links \cite{LS-split}, and the torus link $T(2,6)$  \cite{Gage:T26}.

In this paper, we classify all the links $L$ such that $\rank_{\bZ/2}\Kh(L;\bZ/2)\le 8$, and all the 3-component links $L$ such that $\rank_{\bZ/2}\Kh(L;\bZ/2)\le 12$. Since the rank of the Khovanov homology does not depend on the orientation, it makes sense to refer to $\rank_{\bZ/2}\Kh(L;\bZ/2)$ without orienting $L$.

Let $\Khr(L;R)$ be the reduced Khovanov homology of $L$ with coefficient ring $R$.
By \cite[Corollary 3.2.C]{Shu:torsion_Kh}, we have
 $$\rank_{\bZ/2} \Kh(L;\bZ/2) = 2\rank_{\bZ/2}\Khr(L;\bZ/2).$$
 Moreover, the parity of $\rank_{\bZ/2}\Khr(L;\bZ/2)$ is invariant under crossing changes of $L$. Therefore, $\rank_{\bZ/2}\Kh(L;\bZ/2)$ has the form $4k+2\,\, (k\in \bZ)$ if $L$ is a knot, and is a multiple of $4$ if $L$ is a link with at least two components. 
 
 On the other hand, it is well-known that if $L$ is a link with $n$ components, then $\rank_{\bZ/2}\Khr(L;\bZ/2)\ge 2^n$ (see, for example, \cite[Equation (1)]{XZ:forest}). As a consequence, if $L$ is a knot such that $\rank_{\bZ/2}\Khr(L;\bZ/2)\le 8$, then $\rank_{\bZ/2}\Khr(L;\bZ/2)=2$ or $6$, hence $L$ is an unknot or a trefoil by \cite{KM:Kh-unknot,BS}. If $L$ is a  2-component link such that $\rank_{\bZ/2}\Khr(L;\bZ/2)\le 8$, then $\rank_{\bZ/2}\Khr(L;\bZ/2) =4$ or $8$. If $L$ is a 3-component link such that $\rank_{\bZ/2}\Khr(L;\bZ/2)\le 12$, then $\rank_{\bZ/2}\Khr(L;\bZ/2) =8$ or $12$. If $L$ has at least 4 components, then $\rank_{\bZ/2}\Khr(L;\bZ/2) \ge 16$. In \cite{XZ:forest}, the authors have classified all the $n$--component links $L$ with $\rank_{\bZ/2}\Khr(L;\bZ/2)=2^n$. Therefore, the new content of this paper is given by the following two results:
 
 \begin{Theorem}\label{thm_L4a1}
Suppose $L$ is a 2-component link in $S^3$, then
$$
\rank_{\bZ/2}\Kh(L;\mathbb{Z}/2)=8
$$ 
if and only if $L$ is isotopic to the link L4a1 in the Thistlethwaite link table, which is the link given by Figure \ref{fig_L4a1}, or its mirror image.
\end{Theorem}

\begin{remark}
	In \cite[Corollary 1.4]{XZ:forest}, the authors proved that Khovanov homology (together with the bi-grading) distinguishes an  oriented link whose underlying un-oriented link is isotopic to L4a1. Theorem \ref{thm_L4a1} is a stronger version of that result. 
\end{remark}

\begin{figure}
\centering
\begin{minipage}{.5\textwidth}
  \centering
  \includegraphics[width=.6\linewidth]{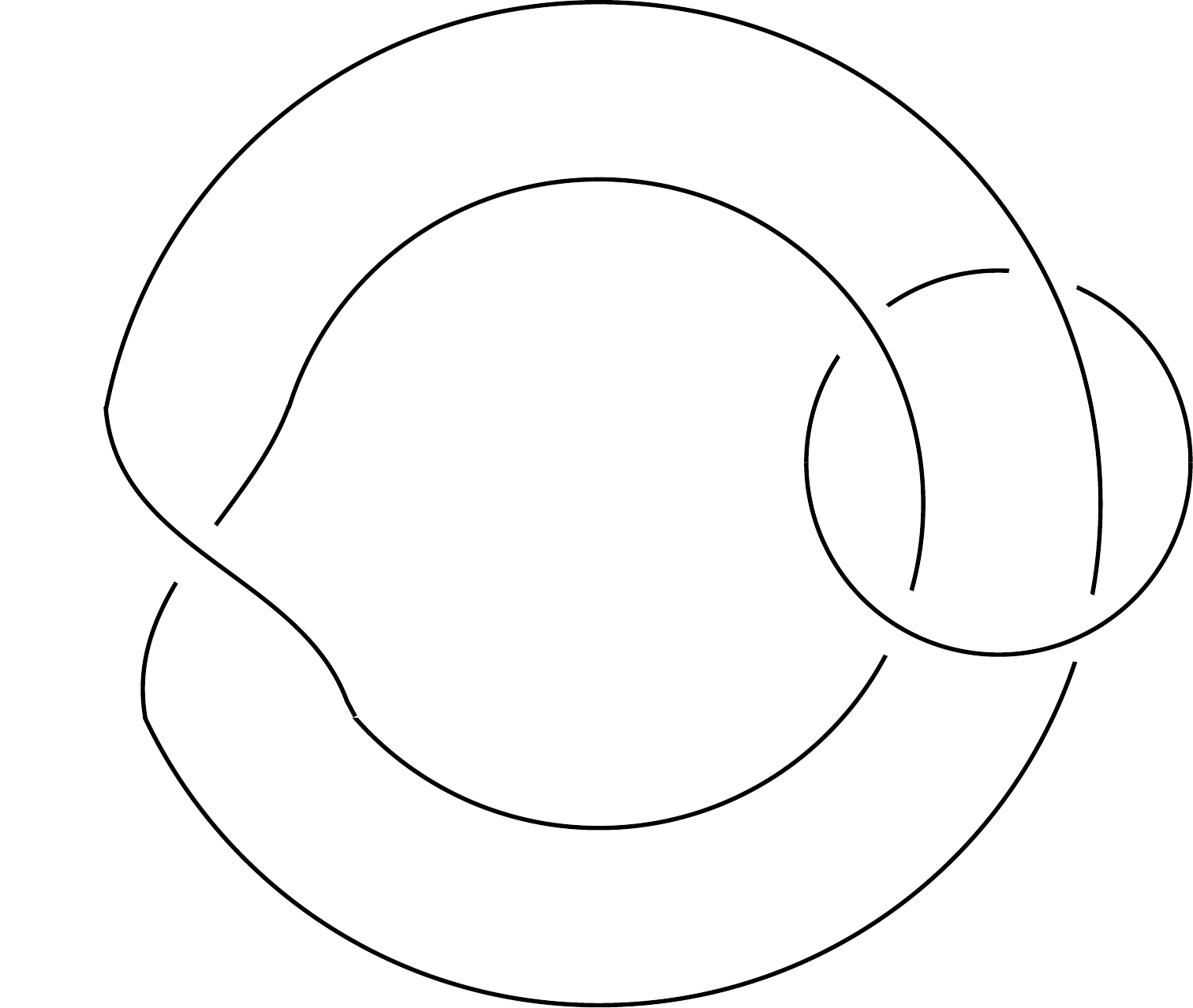}
  \captionof{figure}{The link L4a1.}
  \label{fig_L4a1}
\end{minipage}%
\begin{minipage}{.5\textwidth}
  \centering
  \includegraphics[width=.6\linewidth]{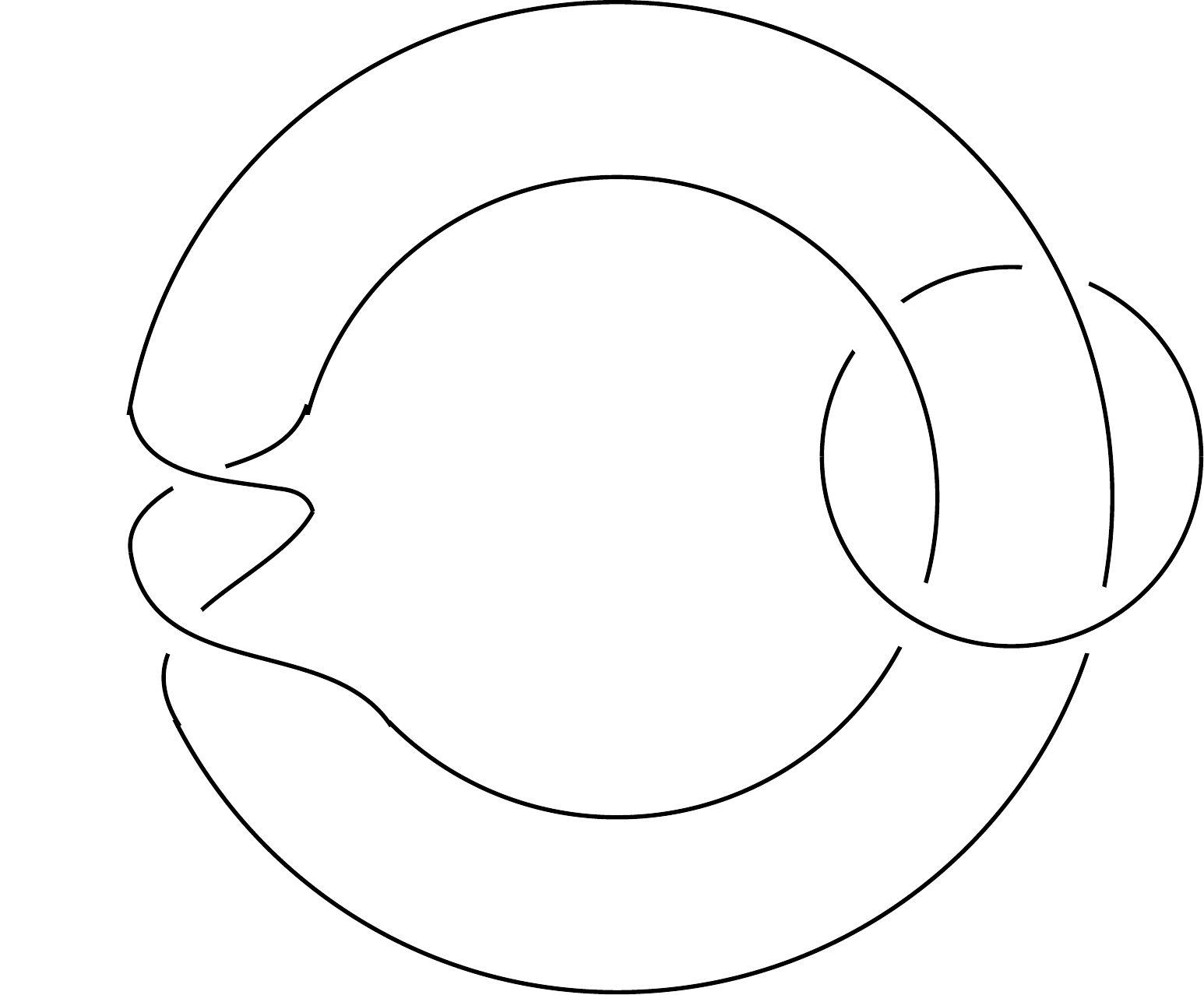}
  \captionof{figure}{The link L6n1.}
  \label{fig_L6n1}
\end{minipage}
\end{figure}

\begin{Theorem}\label{thm_L6n1}
Suppose $L$ is a three-component link in $S^3$, then
$$\rank_{\mathbb{Z}/2}\Kh(L;\mathbb{Z}/2)=12$$
 if and only if $L$ is isotopic to the link L6n1 in
the Thistlethwaite link table, which is the link given by Figure \ref{fig_L6n1}, or its mirror image.
\end{Theorem}

Combining Theorem \ref{thm_L4a1}, Theorem \ref{thm_L6n1}, and the results in \cite{KM:Kh-unknot,BS,XZ:forest}, we have the following two corollaries.

\begin{Corollary}\label{cor_list_Kh_<8}
Suppose $L\subset S^3$ is a link such that $\rank_{\bZ/2}\Kh(L;\mathbb{Z}/2)\le 8$, then $L$ is isotopic to
one of the following:
\begin{itemize}
\item the unlink with at most 3 components;

\item the left-handed or right-handed trefoil;

\item the Hopf link;

\item the connected sum of two Hopf links;

\item the disjoint union of a Hopf link and an unknot;

\item the link L4a1 or its mirror image. \qed
\end{itemize}
\end{Corollary}

\begin{Corollary}
	Suppose $L\subset S^3$ is a link with three components  such that $$\rank_{\bZ/2}\Kh(L;\mathbb{Z}/2)\le 12,$$ then $L$ is isotopic to
one of the following:
\begin{itemize}
\item the unlink with 3 components;
\item the connected sum of two Hopf links;
\item the disjoint union of a Hopf link and an unknot;
\item the link L6n1 or its mirror image. \qed
\end{itemize}
\end{Corollary}

\section{Preliminaries}

Let $L$ be a link $L$ in the (framed) solid torus $S^1\times D^2$, its annular instanton Floer homology $\AHI(L)$ is defined
in \cite{AHI}, and the theory is further developed by \cite{XZ:excision,XZ:forest}. 
This section reviews
 several results from \cite{AHI,XZ:excision,XZ:forest} that will be used later.
 
The annular instanton Floer homology is a $\bZ$-graded complex vector space, and the grading is called the f-grading. We use
$
\AHI(L,i)
$
to denote the component of $\AHI(L)$ with f-degree $i$. For each $i\in \bZ$, we have
\begin{equation}\label{eqn_f_grading_symmetric_at_zero}
\AHI(L,i)\cong \AHI(L,-i).
\end{equation}

We recall the following definition from \cite[Definition 1.5]{XZ:excision}.
\begin{Definition}\label{def_meridional_surface}
A properly embedded, connected 
surface $S\subset S^1\times D^2$ is called a \emph{meridional surface} if $\partial S$ is a meridian of $S^1\times D^2$. 
\end{Definition}

We have the following two results.

\begin{Theorem}[{\cite[Theorem 8.2]{XZ:excision}}]\label{Theorem-2g+n}
Given a link $L$ in $S^1\times D^2$, let $S$ be a meridional surface that intersects $L$ transversely. Let $g$ be the genus of $S$,
and let $n=|S\cap L|$. Suppose $S$ minimizes the value of $(2g+n)$ among meridional surfaces that intersect $L$ transversely,
then we have
\begin{equation*}
\AHI(L,\pm(2g+n))\neq 0,
\end{equation*}
and
\begin{equation*}
\AHI(L,i)= 0
\end{equation*}
for all $|i|> 2g+n$.

\end{Theorem}

\begin{Proposition}[{\cite[Corollary 8.4]{XZ:excision}}]\label{AHI-braid-detection}
Let $L$ be a link in $S^1\times D^2$, let $n$ be a positive integer.
Then $L$ is isotopic to the closure of 
a braid with $n$ strands
if and only if the top f-grading of $\AHI(L)$ is $n$ and $\AHI(L,n)\cong \bC.$
\end{Proposition}

If $K\subset S^3$ is a knot, we will use $N(K)$ to denote the open tubular neighborhood of $K$.
 Suppose $K$ is an unknot, then $S^3-N(K)$ is a solid torus. Choose the framing of $S^3-N(K)$ such that the preferred longitude of $S^3-N(K)$ is a meridian of $K$. Then for every link $L$ that is disjoint from $N(K)$, we can take the annular instanton Floer homology $\AHI(L)$ by viewing $L$ as a link in $S^3-N(K)$.  Notice that in this case, a meridional surface of $S^3-N(K)$ induces a Seifert surfaces of $K$ and vice versa.
 
 The following proposition establishes a relation between the annular instanton Floer homology and the reduced Khovanov homology.

\begin{Proposition}\label{prop_Kh>AHI}
Suppose $L\subset S^3$ is a link, $U\subset L$ is a component of $L$ that is an unknot, and let $p\in U$ be a base point on $U$. Let $L_0=L-U$, then $L_0$ is a link in the solid torus $S^3-N(U)$. We have
$$
\rank_{\bZ/2}\Khr(L;\bZ/2)\ge \dim_\bC \AHI(L_0).
$$ 

\end{Proposition}
\begin{proof}
By Kronheimer-Mrowka's spectral sequence \cite[Theorem 8.2]{KM:Kh-unknot}, we have
$$
\dim_{\bC}\II^\natural(L,p;\bC)\le \dim_{\bC}\Khr(\bar{L},\bar{p};\bC),
$$
where $(\bar{L},\bar{p})$ is the mirror image of $(L,p)$, and $\II^\natural$ is the reduced singular instanton Floer 
homology introduced in \cite{KM:Kh-unknot}.
By \cite[Corollary 11]{Kh-Jones}, we have
$$
\dim_{\bC}\Khr(\bar{L},\bar{p};\bC)=\dim_{\bC}\Khr({L},p;\bC).
$$
By the universal coefficient theorem,
$$
\dim_{\bC}\Khr({L};\bC)\le \rank_{\bZ/2}\Khr(L;\bZ/2).
$$
By \cite[Proposition 2.6]{XZ:forest}, 
$\II^\natural(L,p;\bC)\cong \AHI(L_0).$
Therefore the proposition is proved.
\end{proof}

Let $\mathcal{U}_k\subset S^1\times D^2$ be the unlink with $k$ components, and let
$\mathcal{K}_l\subset S^1\times D^2$ be the link given by  $S^1\times \{p_1,\cdots,p_l\}$. 
Let $\mathcal{U}_k\sqcup\mathcal{K}_l$ be the disjoint union of 
 $\mathcal{U}_k$ and $\mathcal{K}_l$ such that $\mathcal{U}_k$ is included in a 3-ball disjoint from 
 $\mathcal{K}_l$. By \cite[Example 4.2, Proposition 4.3]{AHI},
 $$
 \AHI(\mathcal{U}_k\sqcup \mathcal{K}_l)\cong  \bC^{2^k}_{(0)}\otimes (\bC_{(1)}\oplus\bC_{(-1)})^{\otimes l},
 $$
 where the subscripts represent the f-gradings.

\begin{Proposition}[{\cite[Proposition 4.3]{XZ:forest}}]\label{prop_AHI>UkKl}
Suppose $L\subset S^1\times D^2$ is an oriented link such that every component of $L$
has winding number $0$ or $\pm 1$. Assume there are $k$ components with winding number $0$ and $l$ components
with winding number $\pm 1$,  then we have
\begin{equation*}
\dim_{\bC}\AHI(L,i)\ge \dim_{\bC} \AHI(\mathcal{U}_k\sqcup\mathcal{K}_l,i)
\end{equation*}
for all $i\in\mathbb{Z}$.
\end{Proposition}

\begin{Proposition}[{\cite[Section 4.4]{AHI}}]\label{prop_AHI_parity}
Suppose $L_1, L_2$ are two links in $S^1\times D^2$. If $L_1$ and $L_2$ are \emph{homotopic} to each other in $S^1\times D^2$, then
$$
\dim_\bC\AHI(L_1,i)\equiv\dim_\bC\AHI(L_2,i) ~\mod~2
$$
for all $i\in \bZ$.
\end{Proposition}

\section{Proof of Theorem \ref{thm_L4a1}}

Let $l\ge 2$ be an integer, recall that the $l$--strand braid group $B_l$ has the following 
presentation:
\begin{equation*}
B_l=\langle \sigma_1,\cdots,\sigma_{l-1}|\sigma_i\sigma_{i+1}\sigma_i=  \sigma_{i+1}\sigma_{i}\sigma_{i+1},~
      \sigma_i\sigma_j=\sigma_j\sigma_i~(j-i\ge 2)~           \rangle
\end{equation*}
The reduced Burau representation (see \cite{Birman}) is a group homomorphism
\begin{equation*}
\rho:B_l\to GL(l-1,\mathbb{Z}[t,t^{-1}])
\end{equation*}
which maps $\sigma_i$ to 
\begin{equation*}
\begin{pmatrix}
I_{i-2} &   &    &   &\\
        & 1  &  0  & 0  &\\
        & t  & -t  & 1  & \\
        &  0 &  0  &  1 & \\
        &   &    &   & I_{l-i-2}
\end{pmatrix},
\end{equation*}
where the matrix is truncated appropriately when $i=1~\text{or}~l-1$. Notice that 
$$
\det(\rho(\sigma_i))=-t \text{ for all } i,
$$
hence $\det(\rho(\beta))=\pm t^a$ for all $\beta\in B_l$.

\begin{Definition}
	Suppose $\beta\in B_l$ is  a braid, let $U\subset S^3$ be an unknot, let $\hat \beta\subset S^1\times D^2\cong S^3-N(U)$ be the braid closure of $\beta$. Define $U\cup \hat\beta$ to be the union of $U$ and $\hat \beta$ under the standard framing of $S^3-N(U)$.
\end{Definition}
\begin{remark}
	The link L4a1 is isotopic to $U\cup \hat{\sigma}_1$, where $\sigma_1$ is a generator of $B_2$.
\end{remark}

\begin{Theorem}[{\cite[Theorem 3]{Mor-braid}}]\label{Morton-Alexander-braid}
Let $\beta\in B_l$, and let $L=U\cup \widehat{\beta}$. 
Suppose $\hat \beta$ is connected, then the multi-variable Alexander polynomial $\Delta_L(x,t)$ of $L$ is
given by
\begin{equation}
\label{eqn_morton_alexander}
\Delta(x,t)\doteq\det\big(xI-\rho(\beta)(t)\big),
\end{equation}
where $x$ and $t$ are the variables corresponding to $U$ and $\hat{\beta}$ respectively, and the sign ``$\doteq$'' means that the two sides are equal up to a multiplication by $\pm x^a t^b$. 
\end{Theorem}

\begin{remark}
	The ambiguity in the notation ``$\doteq$'' is necessary because
 the multi-variable Alexander polynomial (before normalization) is only well-defined up to a multiplication by $\pm x^a t^b$.
\end{remark}

We also need the following result:
 \begin{Theorem}[{\cite{Torres}}]\label{Torres-multi-ALexander}
Suppose $L=K_1\cup K_2$ is a 2-component link, and let $\Delta_L(x,y)$ be the multi-variable Alexander polynomial of $L$ where $x$ and $y$ are the variables corresponding to $K_1$ and $K_2$ respectively. Then
we have
\begin{equation*}
	\Delta_L(x,1) \doteq \frac{1-x^l}{1-x} \Delta_{K_1}(x),
\end{equation*}
where $\Delta_{K_1}(x)$ is the Alexander polynomial of $K_1$, and $l=|\lk(K_1,K_2)|$ is the absolute value of the linking number of $K_1$ and $K_2$.
\end{Theorem}

The next lemma is an immediate corollary of the results in \cite{Dowlin,OS:HFL}, and is essentially contained in the proof of \cite[Lemma 6.1]{XZ:forest}. We state it here as a separate lemma for future reference.

\begin{Lemma}
\label{lem_rank_of_Kh_and_Alexander}
	Suppose $L$ is a link with $n$ components, let $\Delta_L(x_1,\cdots,x_n)$ be the (multi-variable) Alexander polynomial of $L$. Let $p\in L$ be a base point.
	\begin{enumerate}
		\item If $n=1$, then the sum of the absolute values of the coefficients of 
		 $\Delta_L(x_1)$
		 is less than or equal to $\,\rank_{\bQ}\Khr(L,p;\bQ)$.
		\item If $n\ge 2$, then the sum of the absolute values of the coefficients of 
		 $$(x_1-1)\cdots (x_n-1)\Delta_L(x_1,\cdots,x_n)$$ 
		 is less than or equal to  $2^{n-1}\rank_{\bQ}\Khr(L,p;\bQ)$.	
	\end{enumerate} 
\end{Lemma}

\begin{proof}
We use $\HFK$ and $\HFL$ to denote 
the Heegaard knot Floer homology \cite{OS:HFK,Ras:HFK} and the link Floer homology \cite{OS:HFL} respectively. 
The link Floer homology was originally defined for ${\mathbb{Z}/2}$-coefficients, and was generalized to $\bZ$-coefficients in \cite{Sar-HFL}.
It is known that 
\begin{equation}\label{eqn_rank_HFK=HFL}
	\rank_{\bQ} \HFK(L;\bQ)=\rank_{\bQ}\HFL(L;\bQ),
\end{equation}
but $\HFL(L;\bQ)$ carries more refined gradings.

By \cite[Corollary 1.7]{Dowlin}, we have
\begin{equation}\label{eq_HFK<Kh}
\rank_\bQ \HFK(L;\bQ)\le 2^{n-1}\rank_\bQ \Khr(L;\bQ).
\end{equation}
By \cite[Equation (1)]{OS:HFL}, the multi-graded Euler characteristics of $\HFL(L;\bQ)$ satisfy 
\begin{equation}
\label{eqn_Euler_characteristic_HFL}
	\chi\big(\HFL(L;\bQ)\big)\doteq 
\begin{cases}
	\Delta_L(x_1) &\text{ if }n=1, \\
	(x_1-1)\cdots (x_n-1)\Delta_L(x_1,\cdots,x_n) &\text{ if }n\ge 2,
\end{cases}
\end{equation}
hence the result is proved.
\end{proof}

Now let $l\ge 2$ be an integer, let $\beta\in B_l$, $L=U\cup \hat \beta$, and let $\Delta_L(x,y)$ be the multi-variable Alexander polynomial of $L$ such that $x$ and $y$ are the variables corresponding to $U$ and $\hat \beta$ respectively.
By \eqref{eqn_morton_alexander}, we have
    \begin{align}
    \Delta_L(x,y)&\doteq (-1)^{l-1}\det(\rho(\beta_2)(y))+ f_1(y)x+\cdots+f_{l-2}(y)x^{l-2}+ x^{l-1}  
    \nonumber
    \\ 
    &= \pm y^a + f_1(y)x+\cdots+f_{l-2}(y)x^{l-2}+ x^{l-1}  
                 \label{eq-Delta-xy}
    \end{align} 
 for some $a\in \mathbb{Z},f_i\in \mathbb{Z}[y,y^{-1}]$.
 By Theorem \ref{Torres-multi-ALexander},
 \begin{equation*}
 	 \Delta_L(x,1)\doteq (1+x+x^2+\cdots+x^{l-1})\Delta_{U}(x)=1+x+x^2+\cdots+x^{l-1}.
 \end{equation*}
Therefore in Equation \eqref{eq-Delta-xy}, we have $f_i(1)=1$ for all $i$, and the sign in front of the term $y^a$ is positive.

A 2-component link $K_1\cup K_2$ is called
 \emph{exchangeably braided}, if both $K_1, K_2$ are unknots, and for each $(i,j)\in \{(1,2),(2,1)\}$, the knot $K_i$ is a braid closure with axis $K_j$. The concept of exchangeably braided links
  was introduced and studied by Morton in \cite{Mor-braid}.
If we further assume that $L$ is exchangeably braided, then by symmetry and \eqref{eq-Delta-xy}, we have
\begin{equation}\label{eq-Delta-yx}
 \Delta_L(x,y)\doteq  x^b + g_1(x)y+\cdots+g_{l-2}(x)y^{l-2}+ y^{l-1}
 \end{equation}
 for some $b\in \mathbb{Z}$ and $g_i(x)\in \mathbb{Z}[x,x^{-1}]$ with $g_i(1)=1$.

\begin{Lemma}\label{Lemma_Delta_xt_lower_bound}
Let $L$ be a mutually braided link with linking number $l\ge 3$,
 let $\Delta_L(x,y)$ be the multi-variable Alexander polynomial of $L$. Then the sum of the absolute values of the coefficients of $(x-1)(y-1)\Delta_L(x,y)$ is at least $12$.
\end{Lemma}
\begin{proof}
Let $f_i(y)$ be as in \eqref{eq-Delta-xy}, and set $f_0(y)=y^a$ and $f_{l-1}(y)=1$. Then we have 
\begin{align}
& (x-1)(y-1)\Delta_L(x,y)
\nonumber
\\
\doteq & (y-1)\Big(  -y^a+\big(y^a-f_1(y)\big)x+\cdots+ \big(f_{l-3}(y)-f_{l-2}(y)\big)x^{l-2} 
\nonumber
\\
& \qquad + \big(f_{l-2}(y)-1\big)x^{l-1}+x^l    
\Big) \nonumber 
\\
=& -(y-1)y^a+(y-1)x^l+\sum_{i=1}^{l-1}(y-1)\big(f_{i-1}(y)- f_{i}(y)\big)x^i.
\label{eqn_standard_form_alexander_mutually_braided}
\end{align}

We discuss three cases. If $f_{i-1}=f_i$ for all $i\in\{1,\cdots,l-1\}$, then $a=0$, and \eqref{eq-Delta-xy} gives 
$$\Delta_L(x,y)\doteq 1+x+\cdots +x^{l-1},$$
which contradicts \eqref{eq-Delta-yx} and the assumption that $l\ge 3$.

If there is exactly one element $i\in\{1,\cdots,l-1\}$ such that $f_{i-1}(y)\neq f_{i}(y)$, then by \eqref{eq-Delta-xy}, we have
$$
\Delta_L(x,y)\doteq y^a+y^a x +\cdots +y^a x^{i-1}+ x^{i}+\cdots +x^{l-1},
$$
which also contradicts \eqref{eq-Delta-yx} and the assumption that $l\ge 3$.

If there exist at least two elements $i\in \{1,\cdots,l-1\}$ such that $f_{i-1}(y)\neq f_{i}(y)$, notice that $f_{i-1}(1)-f_i(1)=1-1=0$, hence $(y-1)$ is a factor of $(f_{i-1}(y)-f_i(y))$. Therefore for such an $i$, the sum of the absolute values of the coefficients of
$$
(y-1)(f_{i-1}(y)- f_{i}(y))
$$
is even and strictly greater than 2, hence it is at least 4. The desired result then follows from \eqref{eqn_standard_form_alexander_mutually_braided}.
\end{proof}

The following lemma refines the proof above and obtains a necessary condition for attaining the lower bound in Lemma \ref{Lemma_Delta_xt_lower_bound}. This result will not be used in the proof of Theorem \ref{thm_L4a1}.
 
\begin{Lemma}
\label{Lemma_Delta_xt_lower_bound_sharp}
Suppose $L$ is an exchangeably braided link with linking number $l\ge 3$,
 let $\Delta_L(x,y)$ be the multi-variable Alexander polynomial of $L$. If the sum of the absolute values of the coefficients of $(x-1)(y-1)\Delta_L(x,y)$ is equal to $12$, then $l=3$.
\end{Lemma}
\begin{proof}
We use the same notation as in the proof of Lemma \ref{Lemma_Delta_xt_lower_bound}. 
If the sum of the absolute values of the coefficients of $(x-1)(y-1)\Delta_L(x,y)$ is equal to $12$, then the proof of  Lemma \ref{Lemma_Delta_xt_lower_bound} indicates that there are exactly two elements $i\in\{1,\cdots,l-1\}$ 
such that $f_{i-1}(y)\neq f_i(y)$. Therefore by \eqref{eq-Delta-xy}, there exists $f(y)\in \bZ[y,y^{-1}]$ and $0\le k_1<k_2\le l-2$, such that $f(y)\neq 1$, $f(y)\neq y^a$, and 
$$
\Delta_L(x,y)\doteq y^a(1+\cdots +x^{k_1})+f(y)(x^{k_1+1}+\cdots + x^{k_2})+x^{k_2+1}+\cdots +x^{l-1},
$$
hence 
$$
\Delta_L(1,y)\doteq (1+k_1)y^a+(k_2-k_1)f(y)+(l-1-k_2).
$$
On the other hand, by Theorem \ref{Torres-multi-ALexander}, we have 
$$
\Delta_L(1,y)\doteq 1+y +\cdots +y^{l-1}
$$
Since $l\ge 3$, comparing the two equations above gives $k_2-k_1=1$. Hence
$$
\Delta_L(x,y)\doteq y^a(1+\cdots +x^{k_1})+f(y)x^{k_1+1}+x^{k_1+2}+\cdots +x^{l-1}.
$$
Since $\Delta_L(x,y)\doteq \Delta_L(x^{-1},y^{-1})$, and recall that $f(y)\neq 1$, $f(y)\neq y^a$, we have $l=2m+1$ for $m\in \bZ$, and
$$
\Delta_L(x,y)\doteq y^a(1+\cdots +x^{m-1})+f(y)x^{m}+x^{m+1}+\cdots +x^{2m}.
$$
View $\Delta_L(x,y)$ as a Laurent polynomial of $x$ with coefficients in $\bZ[y,y^{-1}]$, 
the equation above shows that there is at most one power of $x$ whose 
coefficient is not a monomial of $y$. Switching the roles of $x$ and $y$ and repeating the same argument, we 
 conclude that there is at most one power of $y$ in $\Delta_L(x,y)$  whose 
coefficient is not a monomial of $x$. Therefore we must have $m=1$ and $l=3$.
\end{proof}

Combining \eqref{eqn_rank_HFK=HFL}, \eqref{eqn_Euler_characteristic_HFL}, Lemma \ref{Lemma_Delta_xt_lower_bound}, and Lemma \ref{Lemma_Delta_xt_lower_bound_sharp}, we obtain the following corollary, which may be of independent interest.

\begin{Corollary}\label{cor_Kh_bound_linking_number_3}
Suppose $L$ is an exchangeably braided link with linking number $l\ge3$, then $
\rank_{\bQ}\HFK(L;\bQ)\ge 12.
$
Moreover, if 
$\rank_{\bQ}\HFK(L;\bQ)=12,$
then $l=3$. \qed
\end{Corollary}

\begin{proof}[Proof of Theorem \ref{thm_L4a1}]
The ``if'' part of the theorem follows from a straightforward computation. 
Now suppose $L$ is a 2-component link such that 
\begin{equation}
\label{eqn_assumption_rank_8}
	\rank_{\mathbb{Z}/2}\Kh(L;\mathbb{Z}/2)=8,
\end{equation}
 we prove that $L$ is isotopic to L4a1 or its mirror image.

Denote the two components of $L$ by $K_1$ and $K_2$.
Batson-Seed's inequality \cite[Corollary 1.6]{Kh-unlink} gives
$$
\rank_{\mathbb{Z}/2}\Kh(L;\mathbb{Z}/2)\ge \rank_{\mathbb{Z}/2}\Kh(K_1;\mathbb{Z}/2)\cdot
\rank_{\mathbb{Z}/2}\Kh(K_2;\mathbb{Z}/2).
$$
Since $\rank_{\bZ/2}\Kh(K_i;\bZ/2)\ge 2$, we have 
$$
\rank_{\mathbb{Z}/2}\Kh(K_i;\mathbb{Z}/2)\le 4.
$$
By \cite[Corollary 3.2.C]{Shu:torsion_Kh},
$$
\rank_{\mathbb{Z}/2}\Khr(K_i;\mathbb{Z}/2)=\frac12
\rank_{\mathbb{Z}/2}\Kh(K_i;\mathbb{Z}/2)\le 2.
$$
The parity of $\rank_{\mathbb{Z}/2}\Khr(K_i;\mathbb{Z}/2)$ is invariant under crossing changes, therefore $\rank_{\mathbb{Z}/2}\Khr(K_i;\mathbb{Z}/2)$ is odd, hence it has to be 1.
 Kronheimer-Mrowka's unknot detection theorem \cite{KM:Kh-unknot}
then implies that both $K_1$ and $K_2$ are unknots. 

Pick a base point $p\in K_2$. We have
 $$
\rank_{\mathbb{Z}/2}\Khr(L,p;\mathbb{Z}/2)=\frac12 \rank_{\mathbb{Z}/2}\Kh(L;\mathbb{Z}/2)= 4.
$$
By Proposition \ref{prop_Kh>AHI},
\begin{equation}\label{eq_AHI_le_4}
\dim_{\bC} \AHI(K_1)\le \rank_{\mathbb{Z}/2}\Khr(L,p;\mathbb{Z}/2)=4,
\end{equation}
where $K_1$ is viewed as a knot in the solid torus $S^3-N(K_2)$.

If $\lk(K_1,K_2)=0$, then $K_1$ is \emph{homotopic} to the unknot in the solid torus. By Proposition 
\ref{prop_AHI>UkKl} and Proposition \ref{prop_AHI_parity},
 we have
$$
\dim_{\bC} \AHI(K_1;0)\ge \dim_{\bC} \AHI(\mathcal U_1;0)=2,
$$
and
$$
\dim_{\bC} \AHI(K_1;i)\equiv \dim_{\bC} \AHI( \mathcal{U}_1;i) \equiv 0 ~ (\text{mod}~2), \,\, \text{ for all } i.
$$
Therefore, by \eqref{eqn_f_grading_symmetric_at_zero} and \eqref{eq_AHI_le_4}, $\AHI(K_1)$ must be supported at  f-degree $0$. By Theorem \ref{Theorem-2g+n}, this implies $K_2$ bounds a disk that is disjoint from $K_1$, hence $L$ is the disjoint union of $K_1$ and $K_2$, which implies $L$ is the unlink. However, the unlink does not satisfy the assumption \eqref{eqn_assumption_rank_8}, which is a contradiction.
 
Therefore, we have $l=|\lk(K_1,K_2)|>0$, hence $K_1$ is \emph{homotopic} to the closure $\hat{\beta}$ of an $l$-braid $\beta$ in the solid torus $S^3-N(K_2)$. By Proposition \ref{prop_AHI_parity}, we have
$$
\dim_\bC \AHI(K_1;i)\equiv \dim_\bC \AHI( \hat{\beta};i) ~ (\text{mod}~2).
$$
and by (the easy direction of) Proposition \ref{AHI-braid-detection}, we have
  $$
\dim_\bC \AHI(\hat{\beta};\pm l)=1,
$$
$$\dim_\bC \AHI(\hat{\beta};\pm i)=0 \text{ for all } i>l.
$$
Therefore \eqref{eqn_f_grading_symmetric_at_zero}
and \eqref{eq_AHI_le_4} yield that
  $$
\dim \AHI(K_1;\pm l)=1
$$
and
$$
\dim \AHI(K_1; \pm i)= 0 \text{ for all } i>l.
$$
By Proposition \ref{AHI-braid-detection}, this implies $K_1$ is the closure of an $l$-braid in $S^3-N(K_2)$. 
A similar argument shows that $K_2$ is  
the closure of an $l$-braid in $S^3-N(K_1)$. Therefore $L$ is exchangeably braided.

By the universal coefficient theorem, 
$$
\rank_\bQ \Khr(L,p;\bQ)\le \rank_{\bZ/2} \Khr(L,p;{\mathbb{Z}/2})=\frac12 \rank_{\bZ/2} \Kh(L;{\mathbb{Z}/2})= 4.
$$
Therefore by Lemma \ref{lem_rank_of_Kh_and_Alexander} and Lemma \ref{Lemma_Delta_xt_lower_bound}, we have $l\le 2$.

 If $l=1$, then $L$ is the Hopf link, which does not satisfy \eqref{eqn_assumption_rank_8}.
 
If $l=2$, then $L=U\cup \hat\sigma_1^m$, where $\sigma_1\in B_2$ is a generator of the braid group with 2 strands and $m\in \bZ$. Since both components of $L$ are unknots, we have $m=\pm 1$, hence $L$ is isotopic to L4a1 or its mirror image.
\end{proof}

\section{Proof of Theorem \ref{thm_L6n1}}

The ``if'' part of the theorem follows from a straightforward calculation. 
Now suppose $L$ is a 3-component link with $\rank_{\mathbb{Z}/2}\Kh(L;\mathbb{Z}/2)=12$, we prove that $L$ isotopic to $L6n1$ or its mirror image. 

Denote the three components of $L$
by $K_1$, $K_2$, $K_3$.
 By Batson-Seed's inequality \cite[Corollary 1.6]{Kh-unlink}, we have
 $$
\rank_{\mathbb{Z}/2}\Kh(K_i\cup K_j;\mathbb{Z}/2)\cdot 
\rank_{\mathbb{Z}/2}\Kh(K_k;\mathbb{Z}/2)\le \rank_{\mathbb{Z}/2}\Kh(L;\mathbb{Z}/2)=12
 $$
 for all triples $(i,j,k)$ with $\{i,j,k\}=\{1,2,3\}$. Since $\rank_{\mathbb{Z}/2}\Kh(K_k;\mathbb{Z}/2)\ge 2$, we have
 $$
\rank_{\mathbb{Z}/2}\Kh(K_i\cup K_j;\mathbb{Z}/2)\le 6.
 $$

We apply the same parity argument as before.
By \cite[Corollary 3.2.C]{Shu:torsion_Kh}, we have
 $$
\rank_{\mathbb{Z}/2}\Khr(K_i\cup K_j;\mathbb{Z}/2)=\frac12 \rank_{\mathbb{Z}/2}\Kh(K_i\cup K_j;\mathbb{Z}/2)\le 3.
$$
Crossing changes do not change the parity of $\rank_{\mathbb{Z}/2}\Khr(K_i\cup K_j;\mathbb{Z}/2)$, hence $\rank_{\mathbb{Z}/2}\Khr(K_i\cup K_j;\mathbb{Z}/2)$ is even, therefore
 $$
\rank_{\mathbb{Z}/2}\Khr(K_i\cup K_j;\mathbb{Z}/2)\le 2,
$$
and we have
$$
 \rank_{\mathbb{Z}/2}\Kh(K_i\cup K_j;\mathbb{Z}/2)=2\,\rank_{\mathbb{Z}/2}\Khr(K_i\cup K_j;\mathbb{Z}/2)\le 4.
$$
By \cite[Theorem 1.2]{XZ:forest}, $K_i\cup K_j$ is either a Hopf link or an unlink. In particular, $|\lk(K_i,K_j)|=0$ or $1$. Hence after permuting the labels of the components, we may assume that $|\lk(K_3,K_1)|=|\lk(K_3,K_2)|$.

Pick a base point $p\in K_3$. By \cite[Corollary 3.2.C]{Shu:torsion_Kh}, we have
$$
\rank_{\mathbb{Z}/2}\Khr(L,p;\mathbb{Z}/2)=\frac12 \rank_{\mathbb{Z}/2}\Kh(L;\mathbb{Z}/2)=6.
$$
View $K_1\cup K_2$ as a link in the solid torus $S^3-N(K_3)$,  Proposition \ref{prop_Kh>AHI} gives
\begin{equation}\label{eq_AHI_le_6}
\dim_\bC \AHI(K_1\cup K_2)\le \rank_{\mathbb{Z}/2}\Khr(L,p;\mathbb{Z}/2)=6.
\end{equation}

We discuss two cases.
\\

{\bf Case 1.} $|\lk(K_3,K_1)|=|\lk(K_3,K_2)|=0$. Then $K_1\cup K_2$ is \emph{homotopic} to 
the unlink in the solid torus $S^3-N(K_3)$. Hence by Proposition \ref{prop_AHI>UkKl} and 
Proposition \ref{prop_AHI_parity},
we have 
$$
\dim_{\bC} \AHI(K_1\cup K_2,0)\ge \dim_{\bC} \AHI(\mathcal U_2,0)=4,
$$
and
$$
\dim_{\bC} \AHI(K_1\cup K_2,i)\equiv \dim_{\bC} \AHI(\mathcal{U}_2,i)\equiv 0 ~ (\text{mod}~2) \text{, for all } i.
$$
By \eqref{eqn_f_grading_symmetric_at_zero} and \eqref{eq_AHI_le_6}, 
$\AHI(K_1\cup K_2)$ must be supported at f-degree $0$. By Theorem \ref{Theorem-2g+n}, this implies $K_1\cup K_2$ is spit from $K_3$, so $L$ is either the unlink or the disjoint union of a Hopf link
and an unknot. In both cases we have $\rank_{\mathbb{Z}/2}\Kh(L;\mathbb{Z}/2)=8$, which contradicts the assumption.
\\

{\bf Case 2.}  $|\lk(K_3,K_1)|=|\lk(K_3,K_2)|=1$. Recall that the link $\mathcal{K}_2\subset S^1\times D^2$ is defined by $\mathcal{K}_2=S^1\times \{p_1,p_2\}$, and it can be viewed as a link in $S^3-N(K_3)$. The assumption above implies that  $K_1\cup K_2$ is \emph{homotopic} to $\mathcal{K}_2$.  By Proposition \ref{prop_AHI>UkKl} and 
Proposition \ref{prop_AHI_parity}, we have
\begin{equation}\label{eq_AHI_inequality_case_4}
\dim \AHI(K_1\cup K_2;i)\ge \dim \AHI(\mathcal{K}_2;i)
\end{equation}
and
\begin{equation}\label{eq_parity_case_4}
\dim \AHI(K_1\cup K_2;i)\equiv \dim \AHI( \mathcal{K}_2;i) ~ (\text{mod}~2)
\end{equation}
for all $i$.
Recall that 
\begin{equation}\label{eqn_AHI_of_K2}
	\AHI( \mathcal{K}_2)\cong \mathbb{C}_{(2)}\oplus \mathbb{C}_{(-2)}\oplus \mathbb{C}_{(0)}^2,
\end{equation}
where the subscripts represent the f-gradings.
It then follows from \eqref{eqn_f_grading_symmetric_at_zero}, \eqref{eq_AHI_le_6}, \eqref{eq_parity_case_4}, \eqref{eq_parity_case_4}, \eqref{eqn_AHI_of_K2}
 that 
$$
\AHI( K_1\cup K_2)\cong \mathbb{C}_{(2)}\oplus \mathbb{C}_{(-2)}\oplus \mathbb{C}_{(0)}^4.
$$
By Proposition \ref{AHI-braid-detection}, this implies $K_1\cup K_2$ is the closure of a 2-braid $\beta\in B_2$ in the solid torus $S^3-N(K_3)$. Since $K_1\cup K_2$ is either the unlink or the Hopf link,
$\beta$ is either trivial or $\sigma_1^{\pm 2}$, where $\sigma_1$ is a generator of $B_2$. If $\beta$ is trivial, then $L$ is the connected sum of two Hopf links, hence 
$\rank_{\mathbb{Z}/2}\Kh(L;\mathbb{Z}/2)=8$,
which contradicts the assumption. Therefore $\beta=\sigma_1^{\pm 2}$, hence $L$ is isotopic to L6n1 or its mirror image, and the result is proved.

\bibliographystyle{amsalpha}
\bibliography{references}

\end{document}